\documentclass[11pt]{amsart}
\usepackage{hyperref}
\usepackage{amsfonts,mathrsfs,bbm,latexsym,rawfonts,amsmath,amssymb,amsthm,epsfig}
\usepackage{fullpage, setspace, color}

\newtheorem{theorem}{Theorem}[section]
\newtheorem{lemma}[theorem]{Lemma}
\newtheorem{corollary}[theorem]{Corollary}

\newtheorem{thm}{Theorem}[section]

\newtheorem{rem}[thm]{Remark}

\numberwithin{equation}{section}

\newcommand{\al}{\alpha}

\newcommand{\ld}{\lambda}

\newcommand{\de}{\delta}

\newcommand{\ep}{\varepsilon}
\newcommand{\Si}{\Sigma}

\newcommand{\om}{\omega}
\newcommand{\Om}{\Omega}
\newcommand{\ga}{\gamma}
\newcommand{\Ga}{\Gamma}
\newcommand{\ka}{\kappa}
\renewcommand{\th}{\theta}
\newcommand{\Th}{\Theta}

\newcommand{\F}{\mathcal{F}}

\renewcommand{\P}{\mathcal{P}}

\renewcommand{\S}{\mathscr{S}}
\newcommand{\g}{\mathfrak{g}}

\newcommand{\C}{\mathscr{C}}
\newcommand{\D}{\mathbb{D}}
\newcommand{\U}{\mathbb{S}}

\DeclareMathOperator{\Hol}{Hol}

\newcommand{\Real}{\mathbb{R}}
\newcommand{\Integer}{\mathbb{Z}}

\newcommand{\norm}[1]{\Vert#1\Vert}

\def\<{\left\langle} \def\>{\right\rangle}
\def\({\left(} \def\){\right)}
\newcommand{\n}{\nabla}
\newcommand{\p}{\partial}
\newcommand{\pa}{\partial_{\th,\al}}
\newcommand{\na}{\n_A}

\newcommand{\A}{\mathscr{A}}
\renewcommand{\S}{\mathscr{S}}

\begin{document}
\title{Isolated Singularities of Yang-Mills-Higgs fields on surfaces}
\author{Bo Chen}
\address{Academy of Mathematics and Systems Science, Chinese Academy of Sciences, BeiJing,100190, P.R.China}
\email{chenbo@amss.ac.cn}
	
\author{Chong Song}
\address{School of Mathematical Sciences, Xiamen University, Xiamen, 361005, P.R.China.}
\email{songchong@xmu.edu.cn}
	
	
	
\date{\today}
	
\begin{abstract}
We study isolated singularities of two dimensional Yang-Mills-Higgs fields defined on a fiber bundle, where the fiber space is a compact Riemannian manifold and the structure group is a compact connected Lie group. In
general the singularity can not be removed due to possibly non-vanishing limit holonomy around the singular points. We establish a sharp asymptotic decay estimate of the Yang-Mills-Higgs field near a singular point,
where the decay rate is precisely determined by the limit holonomy. Our result can be viewed as a generalization of the classical removable singularity theorem of two dimensional harmonic maps.
\end{abstract}
	
	
\maketitle

\section{Introduction}
	
Suppose $\Sigma$ is a Riemannian manifold, $G$ is a compact Lie group with Lie algebra $\g$, which is endowed with a bi-invariant metric, and $\P$ is a $G$-principal bundle over $\Sigma$. Let $M$ be a Riemannian manifold
admitting a $G$-action, and $\F=\P\times_G M$ be the associated fiber bundle. Suppose there is a generalized Higgs potential $\mu$ which is just a smooth gauge invariant function on $\F$. Let $\S$ denote the space of smooth
sections of $\F$, and $\A$ denote the affine space of smooth connections on $\P$. Then the \emph{Yang-Mills-Higgs(YMH)} functional is defined for a pair $(A,\phi)\in \A\times\S$ by
\begin{equation}\label{de:ymh}
\mathcal{YMH}(A,\phi):=\lVert \n_A\phi\rVert_{L^2}^2+\lVert F_A\rVert_{L^2}^2+\lVert \mu(\phi)\rVert_{L^2}^2,
\end{equation}
where $\n_A$ is the covariant differential induced by $A$ and $F_A$ is the corresponding curvature 2-form. Critical points of the YMH functional are called \emph{YMH fields}, which satisfy the following Euler-Lagrange equation on
$\Sigma$:
\begin{equation}\label{eq:EL}
\begin{cases}
\n_A^*\n_A\phi+\mu(\phi)\cdot \n\mu(\phi)=0,\\
D_A^*F_A+\phi^*\n_A\phi=0.
\end{cases}
\end{equation}
Here $D_A$ is the exterior derivative and $D_A^*$ is its adjoint operator. The term $\phi^*\n_A\phi$ takes its value in the dual space of $\Om^1(\P\times_{ad}\g)$, namely, for all $B\in \Om^1(\P\times_{ad}\g)$, we have
\[ \<\phi^*\n_A\phi, B\>=\<\n_A\phi, B\phi\>.\]

The YMH theory arises from the research of electromagnetic phenomena and plays a fundamental role in modern physics, especially in quantum field theories. In mathematics, it generalizes the pure Yang-Mills theory and
naturally extends the classical harmonic map theory to the gauged setting, which leads to profound applications in both geometry and topology. Indeed, when the fiber space is a point, then the YMH fields reduce to pure
Yang-Mills fields; when the structure group $G$ is trivial and $\mu=0$, the YMH fields are just harmonic maps from $\Sigma$ to the fiber space $M$.
		
In this paper, we study two dimensional YMH fields with isolated singularities, where $\Sigma$ is a Riemann surface. This type of singular YMH fields naturally emerges as the limit of a sequence of two dimensional YMH
fields with finite energy, if we allow the conformal structure of the underlying surface to vary, see~\cite{Song2}.

The problem on removable singularities of harmonic maps and Yang-Mills fields have been extensively studied and the results are by now quite standard. It is well-known an isolated singularity of a two dimensional harmonic
map with finite energy is removable~\cite{SU}, and an isolated singular point of a four dimensional pure Yang-Mills field with finite energy is also removable~\cite{Uhlenbeck}. In other words, a two dimensional harmonic map
(resp. a four dimensional Yang-Mills field) with finite energy defined on a punctured disk can be extended across the singular point to a smooth harmonic map (resp. Yang-Mills field) on the whole disk.  Analogous results
also hold for four dimensional coupled Yang-Mills equations (cf.~\cite{Parker1}).
	
On the other hand, if the singular set has codimension two, then the singularity of a Yang-Mills field is in general not removable due to possible non-trivial holonomy around singular points. Thus the removability of a
codimension two singular set of a pure Yang-Mills field or coupled Yang-Mills fields can only be achieved by assuming the limit holonomy vanishes (cf. \cite{Sibner1, Sibner2, Sibner3, Smith}). In the general case where the
singularity is not removable, Sibner and Sibner~\cite{Sibner4} gave a classification of singular Sobolev connections by their limit holonomy. This result was later reproved by R{\aa}de~\cite{Rade2,Rade3}, where he was able
to give an optimal estimate of singular pure Yang-Mills fields in dimension four.

Since an isolated point on a surface has codimension two, similar obstructions also arise for two dimensional YMH field with point singularities. If we assume the connection is continuous across the singular point, then the
singularity is removable~\cite{Song1}. Similar results were also abtained for minimal YMH fields in the symplectic setting, which are often referred as \emph{symplectic vortices}~\cite{Ott}. However, in general, one can not
expect that the limit holonomy around the singular point to be trivial. In our previous work, we are able to extract, besides the limit holonomy, certain limit data at the singular point(see Theorem 5.5 of \cite{Song2}, and
also Theorem 1.1 of \cite{Mundet-Tian} for symplectic vortices), but the asymptotic behavior of the YMH fields near the singular points are still not clear.

In this paper, inspired by R{\aa}de's work, we achieve a sharp asymptotic decay estimate of two dimensional YMH fields near isolated singular points. Moreover, we show that the decay rate is precisely determined by the
limit holonomy. In particular, if the limit holonomy is identity then our estimate reduces to a $C^1$-bound of the YMH field at the singular point, which directly implies the removability of singularities. Thus our
result provides a new proof and a non-trivial generalization of the classical removable singularity theorem for two dimensional harmonic maps.

\

Our main result can be stated in the following simple setting since the problem is local in nature. Let $\D\subset \Real^2$ be the unit open disk and $\D^*=\D\setminus\{0\}$ be the punctured unit disk. Let $\P$ be a
principal $G$-bundle over $\D^*$ and $\F=\P\times_G M$ be the associated bundle with fiber $M$. Again we denote the space of smooth sections of $\F$ by $\S$ , and the affine space of smooth connections on $\P$ by $\A$. Since we assume $G$ is connected, the bundle $\P$ and $\mathcal{F}$ is actually trivial. Thus under a fixed trivialization, a section $\phi\in \S$ can be identified with a map $u:\D^*\to M$, while a connection $A\in
\A$ is just a $\g$-valued 1-form. Moreover, since the Higgs potential term $\mu$ does not affect on our analysis and main results, we will simply set $\mu=0$.

Then the YMH functional (\ref{de:ymh}) becomes
\[ E(A, u)=\int_{\D^*}(|\n_A u|^2+|F_A|^2)dv\]
and the Euler-Lagrangian equation \eqref{eq:EL} reduces to
\begin{equation}\label{eq:EL1}
	\begin{cases}
	\n_A^*\n_Au=0,\\
	D_A^*F_A+u^*\n_Au=0.
	\end{cases}
	\end{equation}
Thus a pair $(A,u)\in \A\times \S$ is called a YMH field on $\D^*$ with an isolated singularity at the origin if it satisfies equation \eqref{eq:EL1} in $\D^*$.

Let $(r,\th)\in (0,1)\times \U^1$ be the polar coordinate in $\D^*$. For each $r\in (0,1)$, denote the circle of radius $r$ by $S_r=\{x\in \D^*||x|=r\}$ and the punctured disk of radius $r$ by $\D^*_r=\{x\in \D^*|0<|x|\le r\}$. Recall that the holonomy of connection $A$ along $S_r$ is a conjugacy class in $G$, which we denote by $\Hol(A,S_r)$. More precisely, for any $x=(r,0)\in \D^*$ and $y\in \P_x$, if we parallel transport $y$ along $S_r$, then we will end up with another point $y'\in \P_x$ such that $y'=gy$ for some $g\in G$. Then the holonomy is $\Hol(A,S_r)=[g]$, where $[g]$ denotes the conjugacy class of $g$.

Now we are in position to state our main theorem. Note that given a small constant $\ep>0$ and a YMH field $(A,u)$ on $\D^*$ with finite energy, we can always find some $0<r_0<1$ such that the energy of $(A,u)$ on $\D_{r_0}^*$ is smaller than $\ep$.
	
\begin{theorem}\label{TH0}
There exist constants $\ep>0$, and $C_k>0$ only depending on $k$, such that if $(A,u)\in \A\times \S$ is a smooth YMH field on $\D_{r_0}^*$ with isolated singularity at the origin and $E(A,u)\le\ep^2$, then the following hold.
\begin{enumerate}
\item The limit holonomy of $A$ at the origin exists, namely, there exists a constant $\al\in \g$ such that
\[\Hol(A):=\lim_{r\to 0}\Hol(A,S_r)=[\exp{(-2\pi\al)}].\]
\item There exists a gauge such that $A(r,\th)=ad\th$ on $\D_{r_0}^*$, where $a\in C^\infty(\D_{r_0}^*, \g)$ and for any integer $k\ge 0$,
\begin{equation}\label{e:main-thm-a}
\sup_{S_r}r^k|\n^k_A(a-\al)|\le C_k r^2, \quad \forall r\in (0,r_0/2).
\end{equation}
\item There is a constant $\de_\al=\sqrt{C(A)} \in (0,\frac{1}{2}]\cup \{1\}$ such that for any integer $k\ge 1$
\begin{equation}\label{e:main-thm-u}
\sup_{S_r}r^k|\n_A^{k} u|\le C_k E(A,u)^{\frac12}\(\frac{r}{r_0}\)^{\de_\al}, \quad \forall r\in (0,r_0/2).
\end{equation}
Here the constant $C(A)$ is the Poincar\'e constant explicitly given by (\ref{PC}) below, which is uniquely determined by the limit holonomy $\Hol(A)$.
\end{enumerate}
\end{theorem}
	
\begin{rem}
\begin{enumerate}
\item The existence of limit holonomy $\Hol(A)$ can be guaranteed under weaker assumptions. In fact, it suffices to assume $F_A\in L^p(\D^*)$ for some $p>1$, see Section~\ref{ss: holonomy} below.
\item Actually, we obtain a more refined decay estimate in terms of the curvature $F_A$, see Theorem~\ref{MTH} below.
\item Obviously, Theorem~\ref{TH0} also holds for the symplectic vortex in the symplectic setting, which is just a special class of YMH fields.
\item In Section~\ref{s:optimal-example}, we provide a simple example which demonstrates that the decay estimates in Theorem~\ref{TH0} are optimal.
\end{enumerate}
\end{rem}

An easy corollary of Theorem \ref{TH0} is that the singularity is removable if the limit holonomy is identity.
\begin{corollary}
Suppose $(A,u)\in \A\times \S$ is a YMH field with finite energy, which satisfies equation \eqref{eq:EL1} over $\D^*$. If the limit holonomy $\Hol(A)=id$, then $(A,u)$ can be extend across the origin to a smooth YMH field on
the disk $\D$.
\end{corollary}
\begin{proof}
Since $\Hol(A)=id$, we may set the constant $\al=0$. It follows from the definition (\ref{PC}) that the Poincar\'e constant is $C(A)=1$.

Now by (2) of Theorem~\ref{TH0}, there is a gauge transformation such that $A=ad\th$ and $\lim_{r\to0}a=0$. Then we can extend the connection $A$ to the whole disk $\D$ simply by letting $A(0)=0$ at the origin.
Also, since $\de_\al=1$, (3) of Theorem~\ref{TH0} immediately shows that $\n_Au$ is bounded. In particular, $u$ can be extended across the origin.

Consequently, $(A,u)$ is a continuous YMH field on the whole disk $\D$ with finite energy. The smoothness of $(A,u)$ then follows from standard elliptic theory, see for example~\cite{Song1}.
\end{proof}

We also get an analogous result for twisted harmonic maps with isolated singularities, which serves as an easy version of Theorem~\ref{TH0}.

Recall that, given a flat connection $A$ (but not necessarily trivial, due to non-vanishing holonomy) on the punctured disk $\D^*$, a map $u:\D^*\to M$ is called a \emph{twisted harmonic map} w.r.t. $A$ if it satisfies the
equation
\begin{equation}\label{e:twisted-harmonic-map}
  \tau_A(u)=\n^*_A\n_Au=0.
\end{equation}
Obviously, twisted harmonic maps are critical points of the energy functional $E_A(u)=\int|\n_Au|^2dv$, which is a natural generalization of the familiar harmonic map in a gauged setting. The twisted harmonic map naturally
emerges as a new type of bubbles during the blow-up process of a sequence of two dimensional YMH fields on degenerating Riemann surfaces. In fact, if the energy concentration occurs at an annulus in the collar area of the
degenerating Riemann surface, then it give rise to the so-called ``connecting bubbles" after suitable rescaling. This type of bubbles turns out to be twisted harmonic maps instead of usual harmonic maps, again due to
possibly non-trivial holonomy along the shrinking geodesic. See~\cite{Song2} for more details.

\begin{thm}\label{t:twisted-harmonic-map}
There exist constants $\ep>0$, and $C_k>0$ depending only on $k\ge 0$, such that if $u:\D^*_{r_{0}}\to M$ is a twisted harmonic map w.r.t. a flat connection $A=\al d\th$ on $\D^*_{r_0}$ with energy $E_A(u)\le \ep^2$, then for any integer $k\ge 0$,
\begin{equation}\label{e:thm-twisted-u}
 \sup_{S_r}r^k|\n_A^k u|\le CE_A(u)^{\frac12}\(\frac{r}{r_0}\)^{\de_\al}, \quad \forall r\in (0,r_0/2),
\end{equation}
where the constant $\de_\al^2=C(A)$ is again the Poinc\'{a}re constant decided by $\Hol(A)$.
\end{thm}

\begin{rem}
Obviously, when the connection $A$ is trivial, we recover the classical removable singularity theorem for two dimensional harmonic maps.
\end{rem}

\

For better illustration of the main ideas in our proof, we will first prove Theorem~\ref{t:twisted-harmonic-map} in Section \ref{s: twisted harmonic
map} and then Theorem~\ref{TH0} in Section \ref{s:Decay-cur-section}.

Theorem~\ref{t:twisted-harmonic-map} is proved in three steps. First we conformally change the punctured disk $\D^*$ to an infinitely long half cylinder $\C=[0,+\infty)\times \U^1$ and rewrite equation
(\ref{e:twisted-harmonic-map}) in an extrinsic form. Next we derive a second order differential inequality of the angular energy of $u$, i.e.
\[ \Th(t):=\int_{\{t\}\times\U^1}|\p_{\th,\al}u|^2d\th, t\in [0,+\infty),\]
where $\p_{\th,\al}=\p_\th+\al$ is the operator induced by $\al$. Then we deduce an exponential decay estimate of $\Th(t)$ by simple comparison principals. Finally we obtain the exponential decay estimate of the radial energy of $u$ by using the Pohozaev identity, which, together with the $\ep$-regularity theorem, yields the desired estimate~\eqref{e:thm-twisted-u}.

The above strategy is quite standard in blow-up analysis of two dimensional harmonic maps and perhaps is well-known to experts. However, there are two technical issues we need to address in order to obtain \emph{sharp}
estimates. The first one is that we need to explicitly determine the Poincar\'e constant $C(A)$. The second one is R{\aa}de's observation~\cite{Rade1} that instead of deriving an equation for the angular energy
$\Th(t)$ in Step 2, we should consider its square root $\ga(t):=\sqrt{\Th(t)}$.

Theorem~\ref{TH0} is proved in a similar but more involved manner, since we also need to estimate the connection $A$, which is no longer flat. This is accomplished by using a bootstrapping technique for the coupled
system~\eqref{eq:EL1}. Namely, we start with a preliminary estimate of $u$ given by the $\ep$-regularity theorem, then we derive an decay estimate of $A$, which in turn improve the estimate of $u$. Note that since $A$ is not flat, the Poincar\'e constant w.r.t. $A(r,\cdot)$ on each circle $S_r$ is in general not uniformly bounded as $r\to 0$. We overcome this technical issue by simply using the Poincar\'e inequality of the limit connection around the origin.

\

The rest of our paper is organized as follows. In Section 2, we recall the setting of YMH theory and some preliminary lemmas. In Section 3, we establish the generalized Poincar\'e inequality with connections on $\U^1$, where the best constant is explicitly determined. In Section 4, we prove Theorem~\ref{t:twisted-harmonic-map} for twisted harmonic maps with isolated singularities. In Section 5, we prove Theorem~\ref{TH0} for YMH fields with
isolated singularities. Finally in Section 6, we construct an explicit example with optimal decay rate, showing that our results are sharp.

\section{Preliminaries}

\subsection{Yang-Mills-Higgs functional and Euler-Lagrange equation}

Let $(M,h)$ be a compact Riemannian manifold, $G$ be a compact and connected  Lie group with Lie algebra $\g$.  Suppose $M$ supports an action of $G$, which preserves the metric $h$. Let $\D^{*}=\D\setminus\{0\}$ denote the
punctured disk in the two dimensional Euclidean space. Let $\P$ be a $G$-principal bundle over $\D^{*}$ and $\mathcal{F}=\P\times_{G}M$ be the associated bundle.

Let $\S:=\Ga(\mathcal{F})$ denote the space of smooth sections of $\mathcal{F}$ and $\A$ denote the space of smooth connections  which is an affine space modeled on $\Omega^{1}(\P\times_{Ad}\mathfrak{g})$. A connection $A\in \A$ naturally induces a covariant derivative $\n_A$ on $\F$ and an exterior derivative $D_A$ on $\P\times_{Ad}\mathfrak{g}$. The curvature of $A$ is defined by $F_{A}=D_A^2\in
\Omega^{2}(\P\times_{Ad}\mathfrak{g})$.

We define the Yang-Mills-Higgs(YMH) functional of a pair $(A,\phi)\in \A\times \S$ by
$$\mathcal{YMH}(A,\phi):=\int_{\D^{*}}| F_{A}|^{2}dv+\int_{\D^{*}} | \n_{A}\phi|^{2}dv.$$
Let $\mathcal{G}:=Aut(\P)=\P\times_{Ad}G$ be the gauge group of $\P$. Under a gauge transformation $s \in \mathcal{G}$, the connection $A$ and its curvature $F_A$ transform by the following law
$$s^{*}A=s^{-1}ds+s^{-1}As,\, s^{*}F_{A}=F_{s^{*}A}=s^{-1}F_{A}s.$$
Obviously, the Yang-Mills-Higgs functional is invariant under gauge transformations, that is
$$\mathcal{YMH}(A,\phi)=\mathcal{YMH}(s^{*}A,s^{*}\phi), \,\forall s\in \mathcal{G}.$$

The critical points of the YMH functional are called YMH fields which satisfy the Euler-Lagrangian equations
\begin{equation}\label{eq:EL2}
\begin{cases}
\n^{*}_{A}\n_{A}\phi=0,\\[1ex]
D^{*}_{A}F_{A}=-\phi^{*}\n_{A}\phi,
\end{cases}
\end{equation}
where $D_A^*$ and $\n_A^*$ are the dual operators of $D_A$ and $\n_A$ respectively, and the term $\phi^*\n_A\phi$ lies in the dual space of $\Om^1(\P\times_{ad}\g)$, namely, for all $B\in \Om^1(\P\times_{ad}\g)$, we have
\[ \<\phi^*\n_A\phi, B\>=\<\n_A\phi, B\phi\>.\]

Next we rewrite the Euler-Lagrangian equation more explicitly in a local trivialization. Since $G$ is connected and $\D^{*}$ is homeomorphic to $\U^{1}\times \Real^{1}$, the bundles $\P$ and $\F$ can be trivialized, i.e.
$\P=\D^*\times G$, $\F=\D^*\times M$. Thus any section $\phi \in \S$ can be identified with a smooth map $u:\D^{*}\longrightarrow M$, and any connection $A\in \A$ can be written as $A=A_{r}dr+A_{\theta}d\theta$, where
$A_{r}$ and $A_{\theta}$ are in $\mathcal{C}^{\infty}(\D^{*}, \mathfrak{g})$. Then the induced covariant derivative has the form $\n_{A}=\n+A$, such that
$$\n_{A}\phi:=\n u+A.u=du+A_{r}.udr+A_{\theta}.ud\theta,$$
where $.$ denote the infinitesimal action of $\g$ on $M$.

Thus the Euler-Lagrange equation \eqref{eq:EL2} is equivalent to
\begin{equation}\label{eq:EL3}
\begin{cases}
\n^{*}_{A}\n_{A}u=0,\\[1ex]
D^{*}_{A}F_{A}=-u^{*}\n_{A}u.
\end{cases}
\end{equation}

To better understand the equation \eqref{eq:EL3}, we need an explicit expression of the infinitesimal action of $\g$ on $M$. For $\forall a\in \g$, let $\varphi_{s}=\exp(s a): M\longrightarrow M$ be the $1$-parameter
group of isomorphism generated by $a$. Then $a$ induces a vector field $X_{a}\in \Ga(TM)$ by
$$a.u:=\frac{d}{ds}\Big|_{s=0}\varphi_s(y)=X_{a}(y).$$
Similarly, $a$ acts on a vector field $V\in \Ga(TM)$ by
$$a.V:= \frac{\nabla}{ds}\Big|_{s=0}(\varphi_{s})_{*}(V)=\nabla_{V} X_{a}.$$
where $\nabla$ is the Levi-Civita connection on $M$. Since the $G$-action preserves the metric $h$ on $M$,  $X_{a}$ is a Killing field. Thus  $\nabla X_{a}$ is skew-symmetric, i.e.
$$h(\nabla _{V}X_{a},W)=-h(\nabla _{W}X_{a},V),\quad V,\,W \in \Ga(TM).$$

Let $\tau(u)=\n^*\n u$ denote the tension field of map $u$. A direct calculation shows that equation \eqref{eq:EL3} is equivalent to
\begin{equation}\label{eq:EL4}
\begin{cases}
\tau(u)-d^{*}A.u+2A.du+A^{2}.u=0,\\[1ex]
d^{*}dA+[A,dA] +[A,[A,A]]=-u^{*}(\n_{A}u).
\end{cases}
\end{equation}

For the purpose of PDE analysis, we also need an extrinsic form of \eqref{eq:EL4}. First we recall the following equivarant embedding theorem by Moore and Schlafly~\cite{MS}.
\begin{theorem}\label{TH1}
Suppose $M$ is a compact Riemannian manifold and $G$ is a compact Lie group which acts on $M$ isometrically, then there exists an orthogonal representation $\rho:G\longrightarrow O(K)$ and an isometric embedding
$i:M\rightarrow\mathbb{R}^{K}$ such that $i(g.y)=\rho(g)\cdot i(y)$, for any $y\in M$ and $g\in G$.
\end{theorem}

Since $G$ is connected, we can assume $\rho:G\longrightarrow SO(K)\subset O(K)$. Using this representation, the Lie algebra $\mathfrak{g}$ corresponds to a sub-algebra of $\mathfrak{so}(K)$, i.e. the space of skew-symmetric
$K\times K$ matrices. Thus for any $a\in \mathfrak{g}$ and $y\in M \hookrightarrow \mathbb{R}^{K}$, the infinitesimal action of $a$ on $y$ is simply
$$
a.y=X_{a}(y)=\chi_a\cdot y,
$$
where $\chi_a=\rho(a)\in \mathfrak{so}(K)$. It follows that the action of $a$ on a vector field $V\in \Ga(TM)$ is
$$
a.V=\nabla_{V}X_{a}=(\chi_a\cdot V)^{\top}=\chi_a\cdot V-\Ga(y)(X_{a}, V),
$$
where $\top$ denotes the projection from $\mathbb{R}^K$ to the tangent space and $\Ga$ denotes the second fundamental form.

Using these notations, we can rewrite the equation \eqref{eq:EL4} as
\begin{equation}\label{eq:EL5}
\begin{cases}
\Delta u-d^{*}A.u+2A.du+A^{2}.u=\Ga(u)(\n_{A}u,\n_{A}u),\\[1ex]
d^{*}dA+[A,dA] +[A,[A,A]]=-u^{*}(\n_{A}u).
\end{cases}
\end{equation}

\subsection{Coulomb gauge and epsilon regularity}

Next we recall the basic $\ep$-regularity theorem for YMH fields, starting with the following well-known Uhlenbeck's theorem on Coulomb gauge~\cite{Uhlenbeck}.
\begin{theorem}\label{TH2}
Let $p>1$  and $\P=\D\times G$ be a trivial $G$-principal bundle over $\D$, where $\D$ is the unit disk in $\mathbb{R}^{2}$, and $G$ is connected compact Lie group. Then there exists constants $\varepsilon_{Uh}$ and
$C_{Uh}$, such that for any connection $\tilde{A}\in \A$, if $\lVert F_{\tilde{A}}\rVert_{L^{p}(\D)}\leq \varepsilon_{Uh}$, then $\tilde{A}$ is gauge equivalent by a gauge transformation $s\in W^{2,p}(\D,G)$ to a connection
$A\in\A$ which satisfies
\begin{itemize}
\item[$(1)$] $d^{*}A=0;$
\item[$(2)$] $A(\nu)=0$, where $\nu$ is the outer  normal vector field of boundary $\partial \D$;
\item[$(3)$]  $\lVert A\rVert_{W^{1,p}(\D)}\leq C_{Uh} \lVert F_{A}\rVert_{L^{p}}.$
\end{itemize}	
\end{theorem}
Under a fixed Coulomb gauge where $d^*A=0$, equation \eqref{eq:EL5} becomes a coupled elliptic system
\begin{equation}\label{eq:EL6}
 \begin{cases}
 \Delta u+2A.du+A^{2}.u=\Ga(u)(\n_{A}u,\n_{A}u),\\[1ex]
 \Delta A+ [A,dA] +[A,[A,A]]=-u^{*}(\n_{A}u).
 \end{cases}
\end{equation}
Then following Sacks-Uhlenbeck \cite{SU}, one can easily prove an $\ep$-regularity theorem for YMH fields. For a proof of the following theorem, we refer to Lemma $4.1$ and $4.2$ in \cite{Song1}.

\begin{theorem}\label{TH3}
Let $\D$ be the unit disk in $\mathbb{R}^{2}$ and $\D_{\frac{1}{2}}$ be the disk with radius $\frac{1}{2}$.
Suppose $(A,u)\in \A\times\S$ is  a smooth YMH field with finite energy and $\lVert F_{A}\rVert_{L^{2}}\le \ep_{Uh}$, then under the Coulomb gauge, we have
\begin{itemize}
\item[$(1)$]  For any $1<p<2$,
$$
\lVert A\rVert_{W^{2,p}(\D_{\frac{1}{2}})}\leq C_{p}(\lVert \n_{A}u\rVert_{L^{2}(\D)}+\lVert F_{A}\rVert_{L^{2}(\D)}),
$$
where $C_{p}$ is a constant only depending on $p$.
\item[$(2)$]  There exists a constant $\varepsilon_{0}$, such that if $\lVert \n_{A}u\rVert_{L^{2}(\D)}\leq \varepsilon_{0}$, then for any $1<p<\infty$,
$$
\lVert u-\bar{u}\rVert_{W^{2,p}(\D_{\frac{1}{2}})}+\lVert A\rVert_{W^{2,p}(\D_{\frac{1}{2}})}\leq C_{p}(\lVert \n_{A}u\rVert_{L^{2}(\D)}+\lVert F_{A}\rVert_{L^{2}(\D)}),
$$
where $\bar{u}$ is the mean value of $u$ on $\D$, and $C_{p}$ is a constant only depending on $p$ and $M$. In particular,
 $$
 \sup_{\D_{\frac{1}{2}}}| \n_{A}u|+ \sup_{\D_{\frac{1}{2}}}| F_{A}|\leq C(\lVert \n_{A}u\rVert_{L^{2}(\D)}+\lVert F_{A}\rVert_{L^{2}(\D)}).
 $$
\end{itemize}  	
\begin{rem}\label{r:reg}
By a bootstrap argument, the above theorem actually implies the following higher order estimate
$$
\sup_{\D_{\frac{1}{2}}}|\nabla^{k}_A(u-\bar{u})|+\sup_{\D_{\frac{1}{2}}}|\nabla^{k}_AF_A|\leq C_{k} (\lVert \n_{A}u\rVert_{L^{2}(\D)}+\lVert F_{A}\rVert_{L^{2}(\D)}).
$$
 \end{rem}
\end{theorem}

\subsection{Balanced temporal gauge and equations on cylinder}\label{ss: eq-c}
An important feature of the YMH functional is that, for a conformal metric $g=e^{2v}g_{0}$, where $v$ is a smooth function on $\D^{*}$ and $g_{0}$ is the Euclidean metric, we have
\begin{equation}\label{eq:EL7}
 \mathcal{YMH}(A,u,g)=\int_{\D^{*}} e^{-2v}| F_{A}|^{2}_{g_{0}}dv_{g_{0}}+\int_{\D^{*}}| \n_{A}u|^{2}_{g_{0}}dv_{g_{0}}.
 \end{equation}

We will often perform the conformal transformation from $\D^{*}$ to $\C=(1,\infty)\times \U^{1}$ by
\begin{equation}\label{conf-to-c}
 \varphi: \D^{*}\longrightarrow \C=(1,\infty)\times \U^{1},\,\, (r,\theta)\mapsto (-\log r, \theta)=(t,\theta),
\end{equation}
where the cylinder $\C$ is equipped with the canonical flat metric $g_1=dt^{2}+d\theta^{2}$. Obviously, $g_{0}=e^{-2t}\phi^*g_1$. In view of the conformal property \eqref{eq:EL7}, the Euler-Lagrange equation \eqref{eq:EL3} has the following form w.r.t. the flat metric $g_1$ on cylinder $\C$
\begin{equation}\label{eq:EL8}
\begin{cases}
\n^{*}_{A}\n_{A}u=0,\\[1ex]
D^{*}_{A}(e^{2t}F_{A})=-u^{*}\n_{A}u,
\end{cases}
\end{equation}
where $\ast$ is the Hodge star operator induced by the metric $g$.

Generally, since the cylinder $\C$  is homotopic to $\U^{1}$, there dose not exist a global Coulomb gauge. This motivates us to choose the so-called \emph{balanced temporal gauge}, whose
existence is guaranteed by the following lemma, cf. Lemma $3.2$ in \cite{Song2}.

\begin{lemma}\label{LM1}
Suppose $A$ is a smooth connection on the principal $G$-bundle $\P$ over cylinder $\C$ and $t_{0}$ is a fixed number in $(1,\infty)$. Then there exists a balanced temporal gauge such that $A=a(t,\theta)d\theta$,
where $a:\C\to \mathfrak{g}$ is a smooth  map. Moreover, there exists a constant $\alpha \in \mathfrak{g}$ such that
$$a(t_{0},\theta)=\alpha,\quad \forall\,\theta\in \U^{1}.$$
\end{lemma}

Under such a balanced temporal gauge where $A=a d\theta$, the curvature is simply
$$F_{A}=\partial_{t}a\,dt\wedge d\theta.$$
Then equation \eqref{eq:EL8} becomes
\begin{equation}\label{eq:EL9}
\begin{cases}
 u_{tt}+u_{\theta \theta}-\Ga(u)(\n_{A}u,\n_{A}u)+\partial _{\theta}a.u+2a.u+a^{2}.u=0,\\
 -(2\partial_{t}a+\partial_{t}\partial_{t}a)d\th+(\partial_{\theta}\partial_{t}a+[a,\partial_{t}a])dt
 =-e^{-2t}u^{*}\n_{A}u.
 \end{cases}
 \end{equation}

If we denote the derivative in $\th$-direction by
$$\partial_{\theta,a}u:=\partial_{\theta}u+ a.u,$$
then $\na u=u_tdt+\pa ud\th$ and the first equation of \eqref{eq:EL9} has the following form
\begin{equation}\label{eq:EL10}
\Delta_a u:=u_{tt}+\partial_{\theta,a}^{2}u=\Ga(u)(\n_{A}u,\n_{A}u).
\end{equation}
Moreover, the second equation of \eqref{eq:EL9} can be rewritten as
\begin{equation}\label{e:a}
*\na (e^{2t}\p_ta)=u^*\na u.
\end{equation}

\section{Holonomy and Poincar\'{e} inequality}

In this section, we establish the key analytical tool in this paper, i.e. the Poincar\'e inequality with connection on $\U^1$. The inequality is already proved and played an important role in the blow-up analysis of a
sequence of YMH fields in \cite{Song2}, here we take a step further by exploring the best constant of the Poincar\'e inequality.

\subsection{Holonomy}\label{ss: holonomy}
First we recall the definition of holonomy and some basic facts, which will be useful in the study of Poincar\'e constant.

Let $A=A_{r}dr+A_{\theta}d\theta\in \A$ be a connection on a trivial bundle vector bundle $\D^{*}\times \Real^K$. Let $l_{\theta}:=\{(r,\theta)|0<r<1\}$ be the line of angle $\theta$ and $S_{r}=\{x\in \D^{*}||x|=r\}$ be the circle with radius $r>0$. An orthogonal normal frame $\{e_{i}(r,\theta)\}$ can be obtained by first fixing a frame along the line $l_{0}$  and then extending them by parallel transport around the circle
$S_{r}$ for every $0<r<1$. Suppose $e_{i}(r,2\pi)=g(r)e_{i}(r,0)$ for some $g(r)\in G$, then the holonomy on $S_{r}$ is
$$\Hol(A,r)=[g(r)],$$
where $[g(r)]$ denotes the conjugacy class of $g(r)$ in $G$.

More precisely, if we denote the restricted connection on $S_r$ by $\n_{\theta,A}=d_\th + A_\th d\th$, then we require $\n_{\theta,A}e_{i}(r,\th)=0$ for all $(r,\th)$. Thus by setting
$e_{i}(r,\theta)=g(r,\theta)e_{i}(r,0)$, we get a ordinary differential equation
\begin{equation}\label{eq:EL13}
\begin{cases}
\frac{\partial}{\partial_{\theta}} g(r,\theta)+A_{\theta} g(r,\theta)=0,\\[1ex]
g(r,0)=id.
\end{cases}
\end{equation}
There is unique solution $g(r,\theta)$ of \eqref{eq:EL13} and $\Hol(A,r)=[g(r,2\pi)]$. In particular, if $A=\alpha d\theta$ is a flat connection, where $\alpha\in \mathfrak{g}$ is a constant. Then
$g(r,\theta)=\exp(-\alpha\theta)$ and the holonomy is $\Hol(A)=[\exp(-2\pi\al)]$.

From the above construction, it is easy to see that $\Hol(A,r)$ is invariant under gauge transformation (cf. Lemma $3.1$ in \cite{Sibner4}). Moreover, we have (cf. \cite{Song2},\cite{Sibner4},\cite{Smith})
\begin{theorem}\label{TH6}
For $A\in \A$, if $\lVert F_{A}\rVert_{L^{p}(\D^{*})}\leq C$ for some $p>1$, then the limit holonomy \[ \Hol(A)=\lim_{r\to 0} \Hol(A,r)\]
exists.
\end{theorem}

\subsection{Poincar\'{e} inequality with connection on $\U^1$}\label{ss:poincare}

Recall that for any map $u\in W^{1,2}(\U^{1},\mathbb{R}^{K})$, we have the standard Poincar\'{e} inequality
$$
\int_{\U^{1}} | u-\bar{u}|^{2}\leq \int_{\U^{1}}|\partial_{\theta} u|^{2},
$$
where $\bar{u}$ is the average of $u$ on $\U^{1}$. Following \cite{Song2}, here we prove a generalized Poincar\'{e} inequality with connection $A$ on $\U^{1}$, but with more emphasis on the best constant $C(A)$.

Let $A$ be a connection on a trivial vector bundle $\U^{1}\times \Real^K$, by choosing a gauge similar to the balanced temporal gauge in Lemma \ref{LM1}, we may assume $A=\alpha d\theta$, where $\al\in \mathfrak{g}\subset
\mathfrak{so}(K)$ is constant and $\Hol(A)=[\exp(-2\pi \al)]$. Under a constant gauge transformation if necessary, we may further assume $\al$ has the standard form:
\begin{equation}\label{al}
 \alpha=
\left(
\begin{array}{c|c}
B_{2m\times 2m}& 0_{(K-2m)\times 2m} \\ \hline
0_{2m\times(K-2m)}& 0_{(K-2m)\times (K-2m)}
\end{array}
\right),
\end{equation}
where
\[B=\begin{pmatrix}
0&-a_{1} & & & & \\
a_{1}&0& & &\\
&       0&-a_{2}\\
&        a_{2}&0\\
&         & \ddots \\
&          &     &     0&-a_{m}\\
&          &     &     a_{m}&0\\

\end{pmatrix}.\]

Now define a constant by
\begin{equation}\label{PC}
C(A):= \min_{1\leq j\leq m} \lambda_{j},
\end{equation}
where
$$\lambda_{j}=
\left\{\begin{array}{l}
\min_{k\in \mathbb{Z}}(k+a_{j})^{2},\,  a_{j}\notin \mathbb{Z},\\[1ex]
1,              \quad\quad a_{j}\in \mathbb{Z}.
\end{array} \right.
$$
Note that, without loss of generality, we can always assume $a_{i}\in [0,1)$ for $1\leq i\leq m$. For, if $a_{i}=k+b_i$ where $k\in \Integer$ and $b_i\in [0,1)$, we can modify the gauge by winding the frame $k$ times along
the circle. Actually in this way, the constant $\al$ is uniquely determined by $A$ and vice versa.
Therefore, if $a_{i}\notin\mathbb{Z}$, the constant $\lambda_{i}$ is the minimum of $a^{2}_{i}$ and $(1-a_{i})^{2}$, which belongs to $(0,\frac{1}{4}]$.

The next lemma shows that $C(A)$ is well-defined, i.e. gauge invariant, which implies $C(A)$ is determined by its holonomy $\Hol(A)=[\exp(-2\pi\al)]$.

\begin{lemma}\label{LM4}
Let $A=ad\theta$ be a smooth connection on $\U^{1}$ and $\p_{\th, a}=\p_\th+a$. Then $C(A)$ is the first positive eigenvalue for the elliptic operator $L_{a}=-\partial^{2}_{\theta,a}:
W^{2,2}(\U^{1},\mathbb{R}^{K})\longrightarrow L^{2}(\U^{1},\mathbb{R}^{K})$.
\end{lemma}
\begin{proof}
First we assume that $A=\alpha d\theta$ where $\al$ has the form \eqref{al}. For simplicity, we may assume $K=2m$, the general case where $K>2m$ follows by a similar argument.
A complete orthogonal basis of $L^{2}(\U^{1}, \mathbb{R}^{2m}\cong\mathbb{C}^{m})$ is given by
$$
u_{l}^j(\theta)=(0, \dots, e^{l\sqrt{-1}\theta}, \dots, 0)^{T}, \, 1\leq j\leq m,\, l\in \mathbb{Z},
$$
where the $j$'th element of $u_l^j$ is $e^{l\sqrt{-1}\theta}$. Since
$$
L_{\alpha}(u_{l}^j)=-\partial^{2}_{\theta,\alpha}u_{l}^j=(l+a_{j})^{2}u_{l}^j,
$$
it follows that $\{(l+a_{j})^{2}\}$ are all eigenvalues of operator $L_{\alpha}$ corresponding to eigenfunctions $\{u_{l}^j\}$. It is then easy to verify that the first positive eigenvalue of $L_{\alpha}$ is exactly the constant $C(A)$ defined by (\ref{PC}).
   	
In general, suppose $A=a(\theta)d\theta$, then there exists a balance temporal gauge $s:\U^{1}\to SO(K)$ such that $s^{*}A=\alpha d\theta$  as above. Let $u:\U^{1}\longrightarrow \mathbb{R}^{K}$ be an eigenfunction of
$L_{\alpha}$ such that $L_\al(u)=\ld u$. Since
\[\p_{\th,\al}u=\partial_{\theta,s^{*}a} (u)=s^{-1}\circ\partial_{\theta,a}\circ s(u),\]
it follows
\[L_\al(u)=-\partial^{2}_{\theta,s^{*}a}u=-s^{-1}\circ \partial^{2}_{\theta,a} \circ s(u)=s^{-1}\circ L_a\circ s(u),\]
Thus
\[L_{a}(su)=s\circ L_\al(u)=\lambda su, \]
i.e. $su$ is an eigenfunction of $L_a$ with same eigenvalue $\ld$. Therefore, $L_{a}$ and $L_{\alpha}$ have same eigenvalues.
\end{proof}

Now we state the generalized Poincar\'e inequality. Let $\ker L_{a}$ be the kernel of operator  $L_{a}$, and $(\ker L_{a})^{\bot}$ be its orthocomplement under $W^{1,2}$ norm. Namely,
\[ \ker L_{a}=\ker\partial_{\theta,a}=\{u\in W^{1,2}(\U^{1},\mathbb{R}^{K})|\partial_{\theta,a} u=0\},\]
\[(\ker L_{a})^\bot=W^{1,2}(\U^1,\Real^K)/\ker L_a.\]

\begin{theorem}[Poincar\'{e} inequality]\label{p-inq}
Let $A=ad\theta$ be a smooth connection on a trivial vector bundle $\U^{1}\times \Real^K$. Then for all $u\in(\ker L_{a})^\bot$, there holds
\begin{equation}\label{Poin}
C(A)\int_{\U^{1}}| u|^{2}d\theta\leq\int_{\U^{1}}|\partial_{\theta,a}u|^{2}d\theta.
\end{equation}
\end{theorem}
\begin{proof}
Consider the functional $E_A(u)=\int_{\U^{1}}|\partial_{\theta,a}u|^{2}d\theta$ for $u\in (\ker L_{a})^\bot$, and let
$$
\lambda(A)=\inf_{u\in (\ker L_{a})^\bot\setminus\{0\}} \frac{E_{A}(u)}{\lVert u\rVert^{2}_{L^{2}}}=\inf_{u\in (\ker L_{a})^{\bot}\,\text{with}\,\lVert u\rVert_{L^{2}}=1}E_{A}(u).
$$
We claim that $\lambda(A)=C(A)$, from which \eqref{Poin} follows.

First we show $\ld(A)\le C(A)$, which follows directly from the definitions. Indeed, if $u_1\in W^{2,2}$ is an eigenfunction such that $L_au_0=C(A)u_0$, then
\[ E_A(u_0)=\int_{\U^1}\<\p_{\th,a}u_0,\p_{\th,a}u_0\>d\th=\int_{\U^1}\<u_0,L_au_0\>d\th=C(A)\norm{u_0}_{L^2}.\]

To prove $\ld(A)\ge C(A)$, let $\{v_{n}\}\subset (\ker L_a)^\bot$ be a sequence such that $\lim_{n\to \infty}E_A(v_{n})=\lambda(A)$ and $\lVert v_{n}\rVert_{L^{2}}=1$. Since $\{v_{n}\}$ is bounded in $W^{1,2}$, there exist
a subsequence, still denoted by $\{v_{n}\}$, and some $v_{0}\in (\ker L_a)^\bot$, such that
\begin{align*}
v_{n}\rightarrow v_{0} &\text{\quad weakly~in \, }W^{1,2},\\
v_{n}\rightarrow v_{0} &\text{\quad strongly~in\, }L^{2}.
\end{align*}
which implies $\lVert v_{0}\rVert_{L^{2}}=1$ and
$$E_A(v_{0})\leq \liminf_{n\to \infty}E_A(v_{n})=\lambda(A).$$
Thus $v_0$ is a non-trivial function in $(\ker L_a)^\bot$ and the above inequality is actually an equality.
   	
Therefore, $v_{0}$ is a weak solution to the Euler-Lagrange equation,
$$L_av_0=-\partial_{\theta,a}^2v_{0}=\lambda(A) v_{0}.$$
It follows from standard elliptic theory that $v_{0}\in W^{2,2}$, which finishes our proof.
\end{proof}

Since for any $u\in W^{2,2}(\U^{1},\mathbb{R}^{K})$,  $\partial_{\theta,a} u$ belongs to $(\ker \partial_{\theta,a})^\bot$, we immediately get

\begin{corollary}\label{CO}
For any $u\in W^{2,2}(\U^{1},\mathbb{R}^{K})$, we have
$$C(A)\int_{\U^{1}}| \partial_{\theta,a}u|^{2}d\theta\leq\int_{\U^{1}}|\partial^{2}_{\theta,a}u|^{2}d\theta.$$
\end{corollary}

\begin{rem}
From the definition of $C(A)$ in (\ref{PC}), it is clear that $C(A)$ might tend to $0$ even if $A$ varies in a compact subset. Thus in general there is no uniform Poincar\'e inequality, as already pointed out in
\cite{Song2}.
\end{rem}

\section{Decay estimates of twisted harmonic maps}\label{s: twisted harmonic map}

In this section, we restrict ourselves within a simple setting where the connection is flat. That is, we consider the energy functional $E_A(u)=\int|\n_{A}u|^{2}dv$ where $A$ is a flat connection. The critical points of $E_A(u)$ can be regarded as a generalization of the harmonic maps, which are known as gauged or twisted harmonic maps, satisfying the equation \eqref{e:twisted-harmonic-map}, i.e. $\n^{*}_{A}\n_{A}u=0$. We are concerned with the asymptotic decay estimates of twisted harmonic maps on a punctured disk.

Since $E_A(u)$ is conformal invariant, we can identify the punctured disk $\D_{r_0}^*$ with an infinitely long cylinder $\C=[T_0,+\infty)\times \U^1$, which is endowed with standard metric $g=dt^2+d\th^2$. Suppose $A=\al d\th$  is a
flat connection on a trivial bundle $\C\times M$ and $u:\C\longrightarrow M$ is a smooth map. For later applications on the YMH fields in the next section, let us assume $u$ satisfies a more general equation
\begin{equation}\label{eq:EL11}
\n^{*}_{A}\n_{A}u=f,
\end{equation}
where $f\in \Ga(u^{*}TM)$ is exponentially bounded by
\begin{equation}\label{Decay-of-f}
\sup_{\C_t}(| f|+|\pa f|)\leq C \ep e^{-\ka t},
\end{equation}
where $\C_t=\{t\}\times\U^1$ is the circle at $t\geq T_{0}$, and $\ka>0$ is a constant.

Using the equivariant embedding Theorem~\ref{TH1}, we can write \eqref{eq:EL1} as
\begin{equation}\label{eq:EL12}
u_{tt}+\partial^{2}_{\theta,\alpha}u=\Ga(u)(\n_{A}u,\n_{A}u)+f.
\end{equation}
Then a standard argument yields the following $\ep$-regularity for $u$ (cf. Lemma 4.1 in \cite{Song2}).

\begin{lemma}\label{LM2}
There exists a constant $\varepsilon_{0}>0$ such that if $u$ is a smooth solution to equation \eqref{eq:EL11} with
$$\lVert \n_{A}u\rVert_{L^{2}(P_{t})}\leq \varepsilon_{0},$$
for $t\ge T_0+1$ and $P_{t}=[t-1,t+1]\times \U^{1}$, then
$$\sup_{\C_t}| \n_{A}u|\leq C(\lVert \n_{A}u\rVert_{L^{2}(P_{t})}+\lVert f\rVert_{L^{\infty}(P_{t})}).$$
\end{lemma}

The following theorem gives a sharp decay estimate of $\n_A u$, from which Theorem~\ref{t:twisted-harmonic-map} follows easily.

\begin{theorem}\label{TH5}
Suppose $A=\alpha d\theta$ is a flat connection on $\C$, $u:\C\to M$ is a solution of \eqref{eq:EL12} with energy $E_A(u)=\ep^2 \le \ep_0^2$ and $f$ is exponentially bounded by \eqref{Decay-of-f} with $\ka\ge \frac43$. Then for any $t\ge T_0+1$,
$$\sup_{\C_t}| \n_{A}u|\leq C \ep e^{-\sqrt{C(\al)} (t-T_0)},$$
where $C(\alpha)=C(A)$ is the Poincar\'e constant of $A$ defined by \eqref{PC}.
\end{theorem}
\begin{proof}
By definition $|\n_Au|^2=|\p_t u|^2 + |\p_{\th,\al}u|^2$, we may divide the energy $E_A(u)$ into two parts
\begin{equation}\label{eq:EL15}
\Theta(t)=\int_{0}^{2\pi} | \pa u|^{2}d\theta,\quad H(t)=\frac{1}{2}\int_{0}^{2\pi} |\partial_{t}u|^{2}-| \pa u|^{2}d\theta.
\end{equation}
Since the total energy $E_A(u)$ on the cylinder $\C$ is bounded, it is obvious that $|\n_{A}u|$ vanishes as $t\to \infty$. Hence
\[ \lim_{t\to\infty}\Th(t)=\lim_{t\to\infty}H(t)=0.\]
Moreover, the $\varepsilon$-regularity (i.e. Lemma~\ref{LM2}) together
with the exponential bound \eqref{Decay-of-f} of $f$ yields
$$\sup_{\C_t}| \n_{A}u|\leq C\(E_A(u,P_{t})^{\frac{1}{2}}+\norm{f}_{L^\infty(P_t)}\)\le C\ep,$$

We will prove the theorem in 4 steps as follows.

\noindent\emph{Step 1: exponential decay of $H(t)$.}\\
A simple calculation yields
\[H'=\frac{d}{dt}H=\int_{\U^{1}}\langle u_{t},u_{tt}\rangle-\int_{\U^{1}}\langle\pa u,\partial_{t}\pa u\rangle
=\int_{\U^{1}}\langle u_{t},u_{tt}+\partial^{2}_{\theta,\alpha} u\rangle.\]
Using equation \eqref{eq:EL12} and the exponential bound \eqref{Decay-of-f} of $f$, we get
\begin{equation}\label{e:H1}
H'=\int_{\U^{1}}\langle u_{t}, \Ga(u)(\n_A u,\n_A u)+ f\rangle
=\int_{\U^{1}}\langle u_{t}, f\rangle\leq C\ep^2 e^{-\ka t}.
\end{equation}
It follows that
\begin{equation}\label{eq:H}
| H(t)|\leq \int_{t}^{\infty}| H^{'}(s)| ds\leq C\ep^2 e^{-\ka t}.
\end{equation}

\noindent\emph{Step 2: differential inequality of $\ga=\sqrt{\Th}$.}\\
Taking twice derivatives of $\Th$ and substituting $u_{tt}$ by equation \eqref{eq:EL12}, we get
\begin{align*}
\Th^{''}&=2\int_{\U^{1}} |\partial_{t}\pa u|^{2} +2\int_{\U^{1}} \langle\pa u_{tt},\pa u\rangle\\
&=2\int_{\U^{1}} |\partial_{t}\pa u|^{2} +2\int_{\U^{1}} |\pa^2 u|^{2}\\
&-2\int_{\U^{1}} \langle\Ga(u)(\n_{A}u, \n_{A}u),\pa^2 u\rangle-2\int_{\U^{1}} \langle\pa f,\pa u\rangle,
\end{align*}
For the last two terms, we have
\begin{align*}
\int_{\U^{1}} \<\Ga(u)(\n_{A}u, \n_{A}u),\partial^{2}_{\theta,\alpha} u\>&=\int_{\U^1}\<\Ga(u)(\n_{A}u, \n_{A}u),\Ga(u)(\pa u, \pa u)\>\\
&\le C\sup_{\C_t}|\n_Au|^2\int_{\U^1}|\p_{\th,\al}u|^2,
\end{align*}
and by \eqref{Decay-of-f},
\[\int_{\U^{1}} \langle\pa f,\pa u\rangle\le C\ep e^{-\ka t}\(\int_{\U^1}|\p_{\th,\al}u|^2\)^{\frac12}.\]
Therefore, we arrive at
\begin{equation}\label{eq:Theta}
\Th''\ge 2\int_{\U^{1}} |\partial_{t}\pa u|^{2} +2\int_{\U^{1}} |\partial^{2}_{\theta,\alpha} u|^{2}
-C\sup_{\C_t}|\n_Au|^2\Th-C\ep e^{-\ka t}\Th^{\frac12}.
\end{equation}

Following R{\aa}de~\cite{Rade1}, we set $\gamma^{2}(t)=\Th(t)$. Taking derivative and using H\"older inequality, we get
\[ 2\ga\ga'=\Th'=2\int_{\U^1}\<\p_{\th,\al}u, \p_t\p_{\th,\al}u\>d\th\le 2\ga \(\int_{\U^1}|\pa u|^{2}d\th\)^{\frac12}.\]
It follows
$$(\gamma')^{2}\leq \int_{\U^{1}}|\partial_{t}\pa u|^{2}.$$
Also note that $\Th''=2(\ga')^2+2\ga\ga''$. Inserting into \eqref{eq:Theta}, we obtain
\[\gamma''\gamma\geq \int_{\U^{1}} |\partial^{2}_{\theta,\alpha} u|^{2}-C\sup_{\C_t} |\n_{A}u|^2\gamma^{2}-C\ep e^{-\ka t}\gamma\]
Now apply the Poincar\'{e} inequality in Corollary \ref{CO}, we get
\begin{equation}\label{ODI}
\gamma''\geq (C(\alpha)-C\sup_{\C_t}|\n_Au|^2)\gamma-C\ep e^{-\ka t},		
\end{equation}
where $C(\al)$ is the Poincar\'e constant belonging to $(0,\frac14]\cup \{1\}$.

\noindent\emph{Step 3: exponential decay of $|\na u|$.}\\
Next we apply the comparison principle on the differential inequality \eqref{ODI} to get an exponential bound on $\ga$.
We start by inserting the bound $|\na u|\le C\ep$ into \eqref{ODI}, yielding
$$\gamma''\geq \de^2\gamma-C\ep e^{-\ka t},$$
where $\delta^{2}=C(\alpha)-C\ep^2<1$ and $\ka>1$. Suppose $\gamma(T_{1})=a$ and $\gamma(T_{2})=b$ for $T_{2}\geq T_{1}\geq T_{0}+1$, where by $\ep$-regularity $a,b\le C\ep$. Consider a comparison function
$$g_{0}=C\ep\(2(e^{-\delta(t-T_{1})}+e^{-\delta(T_{2}-t)})-c_0 e^{-t}\),$$
where $c_0= \frac{1}{1-\delta^{2}}$. It is easy to check that $g_0$ satisfies
$$g''_{0}-\delta^{2}g_{0}+C\ep e^{-t}\leq 0,$$
with boundary value
$$g_{0}(T_{1})\geq a,\,g_{0}(T_{2})\geq b.$$
Thus, the comparison principle implies that for all $T_{1}\leq t\leq T_{2}$,
\begin{equation}\label{e:ga1}
\gamma(t)\leq g_{0}(t)\le C\ep (e^{-\delta(t-T_{1})}+e^{-\delta(T_{2}-t)}).
\end{equation}

Now, using the exponential bound \eqref{e:ga1} and \eqref{eq:H} of $H$ from Step 1, we deduce an exponential bound for the energy
$$E_A(u,P_t)= 2\int_{t-1}^{t+1}(H(s)+\ga^2(s))ds\leq C(\ep^2 e^{-\ka t}+g_{0}^2).$$
By the $\ep$-regularity, we can bound $|\na u|$ by
\begin{equation}\label{e:nau1}
\sup_{\C_t}|\n_{A}u|\leq C \(E_A(u,P_t\)^{\frac{1}{2}}+C\ep e^{-\ka t})\leq C(\ep e^{-\frac{\ka}{2}t}+g_{0}).
\end{equation}

\noindent\emph{Step 4: sharp decay estimate of $\n_Au$ by iteration}.\\
To improve the exponential decay estimate to an optimal exponent, we need to iterate the above arguments.

From Step 3, we have the exponential bound \eqref{e:nau1} of $|\na u|$. In view of \eqref{e:H1} in Step 1, we can first improve the bound on $H$ to
\begin{equation}\label{e:H2}
|H(t)|\le \int_t^\infty|u_t(s)||f(s)|ds\le C(\ep^2 e^{-\frac32\ka t}+\ep e^{-\ka t}g_0).
\end{equation}
Next we can rewrite inequality \eqref{ODI} in Step 2 as
\begin{align*}
\gamma^{''}-C(\alpha)\gamma&\geq -C\ep^2g^{2}_{0}-C\ep e^{-\ka t}\\
&\geq -C\ep^2(e^{-2\delta(t-T_{1})}+e^{-2\delta(T_{2}-t)})-C\ep e^{-\ka t}.
\end{align*}
Then we construct another comparison function by
\[ g_{1}(t)=C\ep\(2(e^{-\sqrt{C(\alpha)}(t-T_{1})}+e^{-\sqrt{C(\alpha)}(T_{2}-t)})
-(e^{-2\delta(t-T_{1})}+e^{-2\delta(T_{2}-t)})-c_1 e^{-\ka t}\).\]
where $c_1 =\frac{1}{\ka^2-C(\alpha)}$ and we use the assumption on $\ka$ to make sure $\ka^2>1\ge C(\al)$. One can verify that $g_1$ satisfies
$$
g^{''}_{1}-C(\alpha)g_{1}\leq  \gamma^{''}-C(\alpha)\gamma,
$$
with boundary value
$$
g_{1}(T_{1})\geq a,g_{1}(T_{2})\geq b,
$$
It follows again by comparison principle that
$$\gamma(t)\leq g_{1}(t)\leq C\ep\(e^{-\sqrt{C(\alpha)}(t-T_{1})}+e^{-\sqrt{C(\alpha)}(T_{2}-t)}\).
\eqno{(3.11)}$$
Letting $T_{2}\to \infty$ and $T_{1}=T_0+1$, we obtain
\begin{equation}\label{e:gamma}
\gamma(t)\leq C\ep e^{-\sqrt{C(\alpha)}(t-T_0)}.
\end{equation}

Finally, combining \eqref{e:H2} and \eqref{e:gamma}, we get the sharp exponential decay of the energy
\begin{equation}\label{e:EA}
\begin{aligned}
E_A(u,P_t)&=2\int_{t-1}^{t+1}\(H(s)+\ga^{2}(s)\)ds\\
&\le C\ep^2(e^{-\frac{3}{2}\ka t}+e^{-\ka t}e^{-\de(t-T_0)}+e^{-2\sqrt{C(\alpha)}(t-T_0)})\\
&\le C\ep^2 e^{-2\sqrt{C(\alpha)}(t-T_0)}.
\end{aligned}
\end{equation}
where the last inequality follows from the assumption that $\ka\ge \frac{4}{3}$. Applying the $\ep$-regularity once more, we get
\[ \sup_{\C_t}|\na u| \leq C(E_A(u,P_t)^{\frac{1}{2}}+\ep e^{-\ka t})\le C\ep e^{-\sqrt{C(\alpha)}(t-T_0)}.\]
\end{proof}

Now Theorem~\ref{t:twisted-harmonic-map} is a simple corollay of Theorem~\ref{TH5}.

\begin{proof}[Proof of Theorem\ref{t:twisted-harmonic-map}]
Since the twisted harmonic map $u$ satisfies equation \eqref{eq:EL12} with $f=0$ on the cylinder $\C$, the estimate of $\n_{A}u$ follows form
Theorem \ref{TH5} directly. In fact, in view of the exponential decay of energy \eqref{e:EA}, we can also bound higher order derivatives by $\ep$-regularity theorem (c.f. Remark \ref{r:reg}), i.e.
$$
\sup_{\C_t}|\n^{k}_{A}u|\leq C_{k}E_A(u,P_t)\leq C_k\ep e^{-\de_{\al} (t-T_0)},
$$
where $k\geq 1$ and $\de_\al=\sqrt{C_\al}$.

Now translating from the cylindrical coordinates back to the polar coordinates by the conformal change \eqref{conf-to-c}, we obtain
$$\sup_{S_{r}}r^k|\n^{k}_{A}u|=\sup_{\C_t}|\n^{k}_{A}u|\leq  C_{k}\ep \(\frac{r}{r_0}\)^{\de_{\al}}.$$
\end{proof}

\section{Decay estimates of YMH fields}\label{s:Decay-cur-section}

In this section, we establish exponential decay estimates of YMH fields defined on a punctured disk, by a similar method as the one for twisted harmonic maps in last section. Since the connection is not flat, the argument is considerably more complicated. Indeed, to achieve an optimal estimate, we need to apply a bootstrap argument to the Euler-Lagrangian equation of YMH fields, which is now a coupled system.

Let $(A,u)\in \A\times \S$ be a YMH field defined on the punctured disk $\D_{r_0}^*$ with energy
\begin{equation}\label{e:energy}
  E(A,u)=\int|F_A|^2dv+\int|D_Au|^2dv=: \ep^2\le \min\{\ep_{Uh}^2, \ep_0^2\}.
\end{equation}
After the conformal transformation \eqref{conf-to-c}, we again identify $\D_{r_0}^*$ with the cylinder $\C=[T_0,+\infty)\times \U^1$. Again we denote $P_t=[t-1,t+1]\times \U^1$ and $\C_t=\{t\}\times \U^1$. Then $(A,u)$ satisfies equation \eqref{eq:EL8} on cylinder $\C$. Moreover, for any $t\geq T_{0}+1$,
\begin{equation}\label{Decay-condi}
\lVert F_{A}\rVert_{L^{2}(P_{t})}\leq C\ep e^{-t}, \quad \lVert \n_{A}u\rVert_{L^{2}(P_{t})}\leq \varepsilon.
\end{equation}

For any $t\ge T_0+1$ and $\th\in \U^1$, by \eqref{Decay-condi} and Theorem \ref{TH2}, there exists a Coulomb gauge on a unit ball $\D \subset P_{t}$ centered at $(t,\th)$, such that equation \eqref{eq:EL8} becomes an elliptic system
$$
\left\{\begin{array}{l}
\Delta A+ [A,dA] +[A,[A,A]]-(2\ast dA+\ast[A,A])d\th=-e^{-2t}u^{*}(\n_{A}u),\\[1ex]
\Delta u-\Ga(\n_{A} u,\n_{A}u)+2A.du+A^{2}.u=0.
\end{array} \right.
$$
Then a standard argument yields the following $\ep$-regularity theorem (cf. Theorem \ref{TH3} and Remark~\ref{r:reg}).

\begin{lemma}\label{Decay-cur}
Let $(A,u)$ be a smooth solution to equation \eqref{eq:EL8} on cylinder $\C$ with its energy satisfying condition \eqref{e:energy}. Then for any $k\geq 0$, there exists a constant $C_{k}>0$, such that for all $t\geq T_{0}+1$, there holds
$$\sup_{\C_t}| \nabla^{k}_{A}F_{A}|\leq C_{k} (\lVert F_{A}\rVert_{L^{2}(P_{t})}+e^{-2t}\lVert \n_{A}u\rVert_{L^{2}(P_{t})})\leq C_{k}\ep e^{-t},$$
$$\sup_{\C_t}|\nabla^{k}_{A}u|\leq C_{k} (\lVert F_{A}\rVert_{L^{2}(P_{t})}+\lVert \n_{A}u\rVert_{L^{2}(P_{t})})\le C_k\ep.$$
\end{lemma}

\subsection{Reduced equation}\label{ss: redu-eq}

To obtain decay estimates of $(A,u)$, it is more convenient to work on the balanced temporal gauge which is globally defined on the cylinder. We will choose the gauge such that $A$ is constant on the circle at infinity as below.

By Theorem~\ref{TH6}, the limit holonomy $\Hol(A)=\lim_{t\to \infty}\Hol(A,\C_t)$ exists. Applying Lemma \ref{LM1}, we can find a temporal gauge, such that
\begin{equation}\label{good-gauge}
A=ad\theta,\quad \lim_{t\to \infty} a(t,\theta)=\alpha.
\end{equation}
where $a\in C^\infty(\C,\g)$ and $\al\in\mathfrak{g}$ is a constant such that $\Hol(A)=[\exp(-2\pi\alpha)]$.

Recall that in this gauge, the equation of $u$ has the extrinsic form \eqref{eq:EL10}, i.e.
\[ u_{tt}+\partial^{2}_{\theta,a} u=\Ga(u)(\n_{A}u,\n_{A}u)\]
Setting $\p_{\th,\al}=\p_{\th}+\al$, we can rewrite the equation as
\begin{equation}
\label{eq:EL16}
u_{tt}+\partial^{2}_{\theta,\alpha} u=\Ga(u)(\n_{A}u,\n_{A}u)-f(A,u),
\end{equation}
where
\[ f(A,u)=\p_{\th,a}^2u-\p_{\th,\al}^2u.\]

\begin{lemma}\label{Decay-f}
$f=f(A,u)$ is exponentially bounded by
$$\sup_{\C_t}(| f|+ | \pa f|)\leq C\ep e^{-t}.$$
\end{lemma}

\begin{proof}
A simple computation yields
\[ f=(a-\alpha)^{2}\cdot u+2(a-\alpha)\pa u+\pa a\cdot u.\]
Since in the temporal gauge $A=ad\th$, we have $F_A=\p_ta dt\wedge d\th$ and $\lim_{t\to \infty}a=\al$. By Lemma~\ref{Decay-cur}, the curvature is exponentially bounded by $|F_A|\le C\ep e^{-t}$. Thus
\[ |a-\al|\le \int_{t}^\infty|\p_ta|dt\le C\ep e^{-t}.\]
This together with the bound $|\n_Au|\le C\ep$ gives the exponential bound for the first two terms of $f$.
For the third term, note that
\[ |\partial_{t}\pa a|=|\p_{\th,\al}\p_ta|= |\p_{\th,a}\p_ta-(a-\al)\p_ta|\le |\n_AF_A|+|a-\al||F_A|.\]
Again using  Lemma \ref{Decay-cur}, we have $|\p_t \pa a|\le C\ep e^{-t}$ and $\lim_{t\to \infty}\p_t\pa a=0$. Hence $\pa a$ converges as $t\to \infty$, which implies $a(t,\cdot)$ converges to $\al$ in $C^1(\U^1)$. Therefore,
\[ \lim_{t\to \infty}\pa a=\pa \(\lim_{t\to \infty}a\)=\pa \al=0.\]
Now integrating from $t$ to $\infty$, we get
\[ |\pa a|\le\int_t^\infty|\p_t\pa a|dt \le C\ep e^{-t},\]
and the exponential bound of $f$ follows.

Next, to estimate $\p_{\th,\al} f$, we compute
\begin{align*}
\pa f=&2(a-\al)\pa a \cdot u +(a-\al)^2\pa u+ 3\pa a \pa u \\
& + 2(a-\al)\pa^2u +\pa^2 a\cdot u.
\end{align*}
From previous discussion, we already have the exponential bound for $a-\al$ and $\pa a$, thus the first three terms are also exponentially bounded. The exponential bound for the rest two terms follows similarly from the higher order bound $|\p_{\th,a}^2u|\le C\ep$ and
$$|\p_{\th,a}^2a|\le \int_t^\infty|\na^2F_A|\le C\ep e^{-t}$$
given by Lemma~\ref{Decay-cur}. Thus $\pa f$ is also exponentially bounded as desired.
\end{proof}

\subsection{Proof of the main Theorem \ref{TH0}}

Following a similar but more involved argument as in Theorem~\ref{TH5}, we first prove a sharp decay estimate of $\n_A u$ and $F_{A}$, from which the desired higher order estimates of $(A,u)$ in Theorem~\ref{TH0} will follow easily.

\begin{theorem}\label{MTH}
Let $(A,u)\in \A\times\S$ be  a solution to \eqref{eq:EL8} on cylinder $\C$, which satisfies the energy conditions \eqref{Decay-condi}. Then there exists a smooth map $\al_1:\U^1\to \g$ satisfying $\pa \al_1=0$ such that for $t\geq T_{0}+1$,
\[\begin{aligned}
\sup_{\C_t}| \n_{A}u|&\leq C\ep e^{-\sqrt{C(\al)} (t-T_0)},\\
\sup_{\C_t}| e^{2t}*F_{A}-\al_1|&\leq C \ep e^{-\sqrt{C(\al)} (t-T_0)},
\end{aligned}\]
where $C(\al)$ is the Poincar\'e constant defined by \eqref{PC}.
\end{theorem}

\begin{proof}
Since $u$ satisfies equation~\eqref{eq:EL16} and by Lemma~\ref{Decay-f}, $f$ is exponentially bounded with exponent $\ka=1$, we can apply Step 1-3 in the proof of Theorem~\ref{TH5} to get the exponential decay of the energy
\[   E_A(u,P_t)\le C\ep^2 (e^{-t}+e^{-2\de(t-T_0)}), \]
where $\de^2=C(\al)-C\ep^2$. It follows from the $\ep$-regularity that
\begin{equation}\label{e:decay-of-u-1}
  \sup_{\C_t}|\na u|\le C\ep (e^{-\frac{1}{2}t}+e^{-\de(t-T_0)}).
\end{equation}
	
However, we can not directly follow Step 4 in the proof since it requires exponential decay of $f$ with exponent $\ka>1$. To proceed, we first improve the decay estimate of $A$ and hence of $f$.

Recall that $A=ad\th$ satisfies the second equation of \eqref{eq:EL9} or equivalently \eqref{e:a}, i.e.
$$*\na(e^{2t}\p_ta)=u^*\na u.$$
It follows from \eqref{e:decay-of-u-1} that
\[ |\partial_{t}(e^{2t}\partial_{t}a)|+|\pa (e^{2t}\p_ta)|\le C|\na u|\le C\ep(e^{-\frac{1}{2}t}+e^{-\de(t-T_0)}) .\]
Thus $e^{2t}\p_t a$ converges to some $\al_1:\U^1\to \g$ as $t\to \infty$, such that $\pa \al_1=0$ and
$$|e^{2t}\p_t a-\al_1|\leq C \ep (e^{-\frac{1}{2}t}+e^{-\de(t-T_0)}).$$
Consequently $|F_A|=|\p_t a|\le Ce^{-2t}$ and Lemma \ref{Decay-cur} implies
$$\sup_{\C_t}|\n^{k}_{A}F_{A}|\leq C_{k} e^{-2t},$$
for any $k\geq 0$.

Now a similar argument as Lemma \ref{Decay-f} give an improved decay estimate of $f$ by
$$\sup_{\C_t}(| f|+|\pa f|)\leq C e^{-2t}\le C\ep e^{-\ka t},$$
where $\ka=\frac{4}{3}$ provided $T_0$ is sufficiently large.

Then we can iterate by directly applying Theorem~\ref{TH5} to get the optimal exponential bound of $u$, as desired.
\end{proof}

\begin{rem}
In general, since the limit $\al_1$ is not zero, the decay rate of $|F_A|$ can not be improved to $e^{-(2+\de_{\al})t}$, as is shown by our example in Section 6.
\end{rem}

Now we can complete the proof of Theorem~\ref{TH0}.

\begin{proof}[Proof of Theorem~\ref{TH0}]
By conformal changing $\D^*_{r_0}$ to the cylinder $\C=[T_0,\infty)\times \U^1$, the YMH field $(A,u)$ becomes a solution of equation \eqref{eq:EL8} which satisfies the energy condition~\eqref{Decay-condi}.

By Lemma~\ref{LM1}, we can choose a temporal gauge such that $A=ad\th$ and $\lim_{t\to\infty}a=\al$. Thus the first statement on the existence of limit holonomy $\Hol(A)$ follows directly.

By Theorem~\ref{MTH}, $F_A$ and $\na u$ is exponentially bounded by
\[ \sup_{\C_t}|F_A|\le Ce^{-2t}, \quad \sup_{\C_t}|\na u|\le C\ep e^{-\de_\al(t-T_0)},\]
where $\de_\al^2=C(A)$ is the Poincar\'e constant given by \eqref{PC}.

Then Lemma~\ref{Decay-cur} immediately gives the higher order decay estimates for $k\ge 1$
\[ \sup_{\C_t}|\na^k F_A|\le C_ke^{-2t}, \quad \sup_{\C_t}|\na^k u|\le C_k\ep e^{-\de_\al(t-T_0)}.\]

Finally, the second and third statements follows by translating above decay estimates back to polar coordinates of the disk $\D^*_{r_0}$.
\end{proof}

\begin{rem}		
The proof and hence the same estimates of Theorem~\ref{TH0} also hold true for YMH fields with moment maps $\mu\neq 0$, which satisfies the equation~\eqref{eq:EL}.
\end{rem}

\section{ An optimal example}\label{s:optimal-example}

The purpose of this section is to construct an example to show the estimates in Theorem~\ref{MTH} and hence our main Theorem \ref{TH0} are sharp. We actually construct a minimal YMH field which satisfies a first-order equation.

We first recall the definition of the so-called symplectic vortices. Let $(M,\om)$ be a compact symplectic manifold which supports a Hamiltonian action of a compactly connected Lie group $G$. The moment map is a smooth map $\mu: M\longrightarrow \g$, such that
$$\iota_{X_{\xi}}\om=d\langle\mu, \xi \rangle,\, \forall\xi\in \g,$$
where $X_{\xi}\in \Ga(TM)$ is the vector field generated by $\xi$. Moreover, $\mu$ is equivariant with respect to the adjoint action on $\g$, namely, $\mu(g.y)=g^{-1}\mu(y)g$, for any $g\in G$ and $y\in M$.

Let $\P$ be a $G$-principal bundle over a Riemann surface $(\Si, j)$, and $\F=\P\times_{G}M$ be the associated bundle. Let $J$ be a $G$-invariant $\om$-tamed almost complex structure on $M$. A pair $(A,\phi)\in \A\times \S$ is called a \emph{symplectic vortex} if
\begin{equation}\label{eq:EL17}
\begin{cases}
\bar{\p}_{A}\phi=0,\\[1ex]
*F_{A}+\mu(\phi)=0.
\end{cases}
\end{equation}
Here $*$ is the Hodge star operator and $ \bar{\p}_{A}u:=\frac{1}{2}(\n_{A}u+J\circ\n_{A}u\circ j)$. One can verify that a symplectic vortex satisfies the YMH field equation~\eqref{eq:EL}. In fact, the symplectic vortices are the minimizers of the YMH functional. More details on symplectic vortices can be found in \cite{Mundet-Tian} and \cite{RI}.

Next, we will construct a symplectic vortex with isolated singularity on a trivial bundle over $\D^*$ to show that the decay estimate of YMH field in Theorem \ref{MTH} is optimal. Again we identify $\D^*$ with an infinite cylinder $\C=(0,\infty)\times \U^1$ by a conformal translation. Let the fiber space be the standard sphere $\U^2\hookrightarrow \Real^3$, which supports a action of the group  $G=U(1)\hookrightarrow SO(3)$ by rotation around the $z$-axis. Suppose that $\F=\C\times \U^{2}\hookrightarrow\C\times \mathbb{R}^{3}$ is a trivial fiber bundle over $\C$. Since the Lie algebra is $\g=i\Real^1$, we may choose a connection $A=iad\theta$ where $a=a(t):\C\to \Real^1$ is a smooth function only depending on $t$. Now we consider a symplectic vortex $(A,u)$ which satisfies equation~\eqref{eq:EL17}, or equivalently
\begin{equation}\label{eq:EL18}
\begin{cases}
\p_{t}u=J(u)\p_{\th, ai}u=u\times \p_{\th, ai}u,\\
i\p_{t}a=e^{-2t}\mu(u).
\end{cases}
\end{equation}

Assume $u$ is rotational symmetric and has the form $u(t,\theta)=(\cos\theta \sin f(t), \sin\theta \sin f(t), \cos f(t))$. Recall that the moment map is simply $\mu(u)=i\cos f$. Thus, equation \eqref{eq:EL18} is reduced to
\begin{equation}\label{eq:EL19}
\begin{cases}
f'=-(1+a)\sin f,\\
a'=e^{-2t}\cos f.
\end{cases}
\end{equation}
There is a family of solutions to the first equation given by $f(t)=2\arctan( l e^{-\int_{0}^{t}(1+a(s))ds})$, where $l>0$. Letting $l=1$ and $f(0)=\frac{\pi}{2}$, we get
\begin{equation}\label{eq:EL20}
a'=e^{-2t}\frac{1-e^{-2\int_{0}^{t}(1+a(s))ds}}{1+e^{-2\int_{0}^{t}(1+a(s))ds}}.
\end{equation}
To proceed, we claim that there is a global solution $a$ defined in $(0,+\infty)$ for the ODE \eqref{eq:EL20} with any initial value $a_{0}$.  Actually, by the contraction mapping theory, we can easily obtain a short time solution $a$ of \eqref{eq:EL20}, and the maximal existent time $T$ only depends on the initial value $a_{0}$.  However, the equation \eqref{eq:EL20} shows that $a'$ is uniformly bounded for all $t>0$, which implies $a(t)$ globally exists.

Thus, we obtain a symplectic vortex $(A,u)$, which satisfies
$$| \n_{A}u|=(| u_{t}|^{2}+| \partial_{\theta,ai} u|^{2})^{\frac{1}{2}}=\sqrt{2}| f'(t)|=2\sqrt{2}|
1+a(t)|\frac{e^{-\int_{0}^{t}(1+a(s))ds}}{1+e^{-2\int_{0}^{t}(1+a(s))ds}}.$$
The equation \eqref{eq:EL19} shows $|a'|\le e^{-2t}$. Hence $| a(t)-\al|< \frac{1}{2}e^{-2t}$ for any $t>0$ and $a_{0}\neq -1$. If we take $a_{0}=-1+\ep$, where $\ep$ is a small positive constant, then \eqref{eq:EL19} implies $a'>0$ and
$$-1<a(t)<\al=a_{0}+\int_{0}^{\infty}\p_{t}adt< -\frac{1}{2}+\ep.$$
It follows
$$0<a'\leq e^{-2t}(1-e^{-2(1+\al)t})<e^{-2t}(1-e^{-2(\frac{1}{2}+\ep)t}).$$
Thus,
$$\al=a_{0}+\int_{0}^{\infty}\p_{t}adt<-\frac{1}{2}+(\ep-\frac{1}{3+2\ep})<-\frac{1}{2},$$
for small $\ep>0$.

Therefore, we conclude that $| F_{A}|\approx e^{-2t}$ and
$$|\n_{A}u|\approx e^{\int_{0}^{t}(a(s)-\al)ds}e^{-(1+\al)t}\approx e^{-(1+\al)t},$$
where $\approx$ denotes the both sides have same decay order. By definition~\eqref{PC}, the Poincar\'e constant is $C(\al)=(1+\al)^2$ since $\al\in (-1,-\frac{1}{2}]$. This shows that the decay estimates in Theorem \ref{MTH} can indeed by achieved and hence are optimal.

In particular, if $a(t)=\al$ is constant, then rotational symmetric map $u$ given by $f(t)=2arctan(e^{-\int_{0}^{t}(1+\al)ds})$ is a twisted harmonic map satisfying equation $\n^{*}_{A}\n_{A}u=0$.
A similar argument yields
$$
| \n_{A}u|=2\sqrt{2}\sqrt{C(\alpha)}\frac{e^{-\sqrt{C(\alpha)}t}}{1+e^{-2\sqrt{C(\alpha)}t}},
$$
if we take $\al\in (-1,-\frac{1}{2}]$.  Therefore, the estimates in Theorem \ref{TH5} are also optimal.

\section*{Acknowledgment}
B. Chen would like to express his deep gratitude to Professor Youde Wang for his instructions and encouragements. C. Song is partially supported by the Fundamental Research Funds for the Central Universities (Grant No. 207220170009, 207220180009).

\end{document}